\documentclass[12pt]{amsart}

\usepackage{amsfonts, amsthm, amsmath}

\usepackage{rotating}

\usepackage{tikz}

\usepackage{graphics}

\usepackage{amssymb}

\usepackage{amscd}

\usepackage[latin2]{inputenc}

\usepackage{t1enc}

\usepackage[mathscr]{eucal}

\usepackage{indentfirst}

\usepackage{graphicx}

\usepackage{graphics}

\usepackage{pict2e}

\usepackage{mathrsfs}

\usepackage{enumerate}

\usepackage{tikz}
\usetikzlibrary{tikzmark,arrows.meta}

\usepackage[pagebackref]{hyperref}

\hypersetup{colorlinks=true}

\usepackage{cite}
\usepackage{color}
\usepackage{epic}
\usepackage{hyperref}
\usepackage{framed}
\usepackage{mathabx}
\usepackage{booktabs}
\usepackage{makecell}
\usepackage{comment}
\newcolumntype{V}{!{\vrule width 2pt}}

\numberwithin{equation}{section}
\def\blue{\textcolor{blue}}
\def\red{\textcolor{red}}
\def\green{\textcolor{green}}
\def\magenta{\textcolor{magenta}}

\theoremstyle{plain}

\newtheorem{theorem}{Theorem}[section]

\newtheorem{proposition}[theorem]{Proposition}
\newtheorem{conjecture}[theorem]{Conjecture}
\newtheorem{remark}[theorem]{Remark}
\newtheorem{lemma}[theorem]{Lemma}
\newtheorem{definition}[theorem]{Definition}
\newtheorem{example}[theorem]{Example}

\newtheorem{observation}[theorem]{Observation}

\def\B{\mathcal{B}}
\def\IB{\mathcal{I}}
\def\LB{\mathcal{LB}}
\def\D{\mathcal{D}}
\def\S{\mathfrak{S}}
\def\R{\mathsf{R}}
\def\Rlma{\mathsf{Rma}}

\def\T{\mathcal{T}}

\def\Rlmi{\mathsf{Rmi}}

\def\DES{\mathsf{DT}}

\def\rlmi{\mathsf{rmi}}

\def\rfp{\mathsf{rfp}}
\def\rft{\mathsf{rft}}

\def\N{ \mathbb{N}}

\def\lrrp{\mathsf{lrrp}}
\def\RL{\mathsf{RL}}
\def\RA{\mathsf{RA}}
\def\Tail{\mathsf{Tail}}
\def\s{\mathcal{S}}
\def\lw{\mathsf{lw}}
\def\mlw{\mathsf{mlw}}

\def\llp{\mathsf{llp}}

\def\rrlmi{\mathsf{rrmi}}

\def\c{\mathsf{c}}
\def\r{\mathsf{r}}
\def\lrp{\mathsf{lrp}}

\def\des{\mathsf{des}}
\def\fix{\mathsf{fix}}
\def\bad{\mathsf{bad}}
\def\drop{\mathsf{drop}}

\def\pone{\mathsf{pone}}
\def\prmi{\mathsf{prmi}}
\def\inv{\mathsf{inv}}

\def\rcinv{\mathsf{rcinv}}

\begin{document}

\title[Patterns in $t$-stack sortable permutations]{A  bijection between $321$- and $213$-avoiding permutations\\ preserving  $t$-stack-sortability}

\author[Y. Li]{Yang Li}
\address[Yang Li]{Research Center for Mathematics and Interdisciplinary Sciences, Shandong University \& Frontiers Science Center for Nonlinear Expectations, Ministry of Education, Qingdao 266237, P.R. China}
\email{202421349@mail.sdu.edu.cn}

\author[S. Kitaev]{Sergey Kitaev}
\address[Sergey Kitaev]{Department of Mathematics and Statistics, University of Strathclyde, 26 Richmond Street, Glasgow G1 1XH, United Kingdom}
\email{sergey.kitaev@strath.ac.uk}

\author[Z. Lin]{Zhicong Lin}
\address[Zhicong Lin]{Research Center for Mathematics and Interdisciplinary Sciences, Shandong University \& Frontiers Science Center for Nonlinear Expectations, Ministry of Education, Qingdao 266237, P.R. China}
\email{linz@sdu.edu.cn}

\author[J. Liu]{Jing Liu}
\address[Jing Liu]{Research Center for Mathematics and Interdisciplinary Sciences, Shandong University \& Frontiers Science Center for Nonlinear Expectations, Ministry of Education, Qingdao 266237, P.R. China}
\email{lsweet@mail.sdu.edu.cn}

\date{\today}

\begin{abstract} We construct a bijection between $321$- and $213$-avoiding permutations that preserves the property of  $t$-stack-sortability. Our bijection transforms  natural statistics between these two classes of permutations and proves a refinement of an enumerative conjecture posed by Zhang and Kitaev. This work contributes further to the long-standing line of research on bijections between length-3 pattern avoiding permutations.  Increasing binary trees lie at the heart of our approach.
\end{abstract}


\maketitle

\section{Introduction}
Let $\S_n$ be the set of all permutations of $[n]:=\{1,2,\ldots,n\}$. Given a pattern $p\in \S_k$, a permutation $\pi=\pi_1\pi_2\cdots\pi_n$  is said to avoid the pattern $p$ if there are no indices $ i_1<i_2<\cdots<i_k$ such that the subsequence $\pi_{i_1}\pi_{i_2}\cdots\pi_{i_k}$ is order isomorphic to $p$. For an arbitrary finite collection of patterns $P$, we say that $\pi$ avoids $P$ if 
$\pi$ avoids any $p\in P$. Given a subset $\mathcal{A}\subseteq\S_n$,  let 
$$
\mathcal{A}(P):=\{\pi\in\mathcal{A}: \pi\text{ avoids } P\}.
$$
A classical  enumerative result in pattern avoidance in permutations, attributed to MacMahon and Knuth (see~\cite{Kit2}), asserts that $|\S_n(\sigma)|=C_n$ for
each pattern $\sigma\in\S_3$, where $C_n=\frac{1}{n+1}{2n \choose n}$ is the $n$-th {\em Catalan number}.

For a permutation $\pi=\pi_1 \pi_2 \cdots \pi_n$ with $\pi_i=n$, the {\em stack-sorting operation} $\s$ is recursively defined as follows, where $\s(\emptyset)=\emptyset$:
\begin{equation}\label{def:stack rec}
\s(\pi)=\s(\pi_1 \cdots \pi_{i-1}) \s(\pi_{i+1} \cdots \pi_n) n.
\end{equation}
A permutation $\pi$ is {\em $t$-stack-sortable} if the permutation $\s^t(\pi)$, obtained by applying the stack-sorting operation $t$ times, is the identity permutation. The set of all permutations of length $n$ that are $t$-stack-sortable is denoted by $\S_n^t$. Knuth~\cite{Knu} connected the stack-sorting algorithm with patterns in permutations by  observing that $\S_n^1=\S_n(231)$.  West~\cite{West} characterized $2$-stack-sortable permutations in terms of ``barred patterns'' and conjectured that 
$$
|\S_n^2|=\frac{2(3n)!}{(n + 1)!(2n + 1)!},
$$
which was verified by Zeilberger~\cite{Zei}.
  While \'Ulfarsson~\cite{Ulf} can  characterize  $3$-stack-sortable permutations  in terms of ``decorated patterns'', no explicit formula for  counting them is known. B\'ona~\cite{Bon} proved that the descent polynomials over $t$-stack-sortable permutations are symmetric  and unimodal.  Defant~\cite{Def} developed a decomposition lemma  which leads to the first
polynomial-time algorithm for enumerating $3$-stack-sortable permutations. For more information about $t$-stack-sortable permutations, the reader is referred to the book~\cite{Bon2} by B\'ona.

Despite considerable research on $t$-stack-sortable permutations, the study of pattern avoidance on them has been relatively under-explored. Recently, Zhang and Kitaev~\cite{Zhang} considered length-$3$ patterns on $t$-stack-sortable permutations and posed the following intriguing  enumerative conjecture.

\begin{conjecture}[Zhang and Kitaev~\cite{Zhang}]\label{main conj}
    For any $t\geq1$, we have
    \begin{equation}
     |\S_n^t(321)|=|\S_n^t(213)|.
    \end{equation}
\end{conjecture}

Conjecture~\ref{main conj} can be regarded as a refinement of the aforementioned MacMahon--Knuth result stating that  $|\S_n(213)|=|\S_n(321)|$. Zhang and Kitaev~\cite{Zhang} confirmed  this conjecture for $t\in\{n- 1, n-2, n-3, n-4\}$, while the case of $t = 1$ is known by the work of Simion and Schmidt~\cite{SS}.  Various bijections between $\S_n(213)$ and $\S_n(321)$ have been found in the literature~\cite{EP, Kra,LLY,Reifegerste,SS} (see also~\cite{CK} and~\cite[Chap.~4]{Kit2} for classification of some classical ones), leading to diverse refinements of the MacMahon--Knuth result. However, none of them  would preserve the property of $t$-stack-sortability. The main objective of this paper is to construct a $t$-stack-sortability  preserving bijection between $\S_n(213)$ and $\S_n(321)$, thereby proving Conjecture~\ref{main conj} (with a refinement) and contributing to a long line of research on bijections between length-3 pattern avoiding permutations, a topic first analyzed by Claesson and Kitaev  in~\cite{CK}.

In order to state our result, we need to recall some classical permutation statistics. Given $\pi=\pi_1\cdots\pi_n\in\S_n$, a letter $\pi_i$ is a {\em descent top}  of $\pi$ if $i<n$ and $\pi_i>\pi_{i+1}$. A letter $\pi_i$ is a {\em drop} (resp.,~{\em excedance}, {\em fixed point})  of $\pi$ if $\pi_i<i$ (resp.,~$\pi_i>i$, $\pi_i=i$). A letter $\pi_i$ is a {\em right-to-left minimum} of $\pi$ if $\pi_i<\pi_j$ for all $j>i$.  Following Athanasiadis~\cite{Ath}, a letter $\pi_i$ is said to be a {\em bad point} if $\pi_i$ is a right-to-left minimum of $\pi$ and either  $i=1$ or $\pi_{i-1}<\pi_i$. Denote by $\des(\pi)$ (resp.,~$\bad(\pi)$, $\drop(\pi)$, $\fix(\pi)$) the number of descent tops (resp.,~bad points, drops, fixed points) of $\pi$. 
The first fundamental transformation (cf.~\cite[Sec.~1.3]{St0}) is a bijection on $\S_n$ that transforms the pair  $(\fix,\drop)$ to $(\bad,\des)$. Our main bijection below possesses the same flavor, which proves a refinement of Conjecture~\ref{main conj}. 

\begin{theorem}\label{main:bij}
There exists a $t$-stack-sortability  preserving bijection $\Upsilon: \S_n(321)\rightarrow\S_n(213)$ that transforms the pair  $(\fix,\drop)$ to $(\bad,\des)$. Consequently, 
\begin{equation}
\sum_{\pi\in\S_n^t(321)}r^{\fix(\pi)}x^{\drop(\pi)}=\sum_{\pi\in\S_n^t(213)}r^{\bad(\pi)}x^{\des(\pi)}. 
\end{equation}
\end{theorem}

The construction of $\Upsilon$ is based on two characterizing lemmas for $321$- or $213$-avoiding  $t$-stack-sortable permutations.  Our bijection $\Upsilon$ is a composition of five bijections, four of which are known, while the fifth is an involution on binary trees. As it turns out, $\Upsilon$ transforms more permutation statistics between these  two classes of permutations; see Section~\ref{subsection:2.2}. 
At the heart of our approaches lie the increasing binary trees, whose symmetry reveals the insight of the proof. Further applications of our approaches will also be discussed, which not only settles another related enumerative conjecture of Zhang and Kitaev~\cite{Zhang} but also provides  a bijective proof of a result previously established via the transfer-matrix method by Chow and West~\cite{CW}.

The rest of this paper is organized as follows. In Section~\ref{sec:2}, we present two characterizing lemmas for $321$- or $213$-avoiding  $t$-stack-sortable permutations and construct  the bijection $\Upsilon$ to prove Theorem~\ref{main:bij}. The proofs of the two characterizing lemmas are postponed to Section~\ref{sec:3}, further applications of which are presented in Section~\ref{sec:4}.

\section{Two characterizing lemmas and the construction of $\Upsilon$}
\label{sec:2}
The purpose of this section is to present two characterizing lemmas for $321$- or $213$-avoiding  $t$-stack-sortable permutations which are key to the construction of $\Upsilon$. These two lemmas are based on a classical bijection between permutations and increasing binary trees. 

A {\em binary tree} is a rooted tree in which each node has at most two children: a left child and a right child. We denote by $\B_n$ the set of all binary trees with $n$ nodes. An {\em increasing binary tree} on $n$ nodes is a binary tree labeled precisely  by $[n]$, where each parent node receives a label  smaller than all its children. Let $\IB_n$ denote the set of all increasing binary trees on $[n]$. 

We now recall a classical bijection $\lambda: \S_n \rightarrow \IB_n$  from~\cite[p.~23]{St0}.
Given a word $w=w_1w_2\cdots w_n$ of positive integers with no repeated letters, we can write it as $w=\sigma i\tau$ with $i$ as the smallest letter of $w$.  Set $\lambda(\emptyset)=\emptyset$ and define $\lambda(w)$ recursively as $(\lambda(\sigma), i,\lambda(\tau))$, the increasing binary tree rooted at $i$ with left branch $\lambda(\sigma)$ and right branch $\lambda(\tau)$.  The inverse map $\lambda^{-1}$ can be viewed as projecting the nodes of the increasing binary trees downwards (i.e., reading the labels of the nodes in in-order). See Fig.~\ref{fig lambda} for an example of $\lambda$.

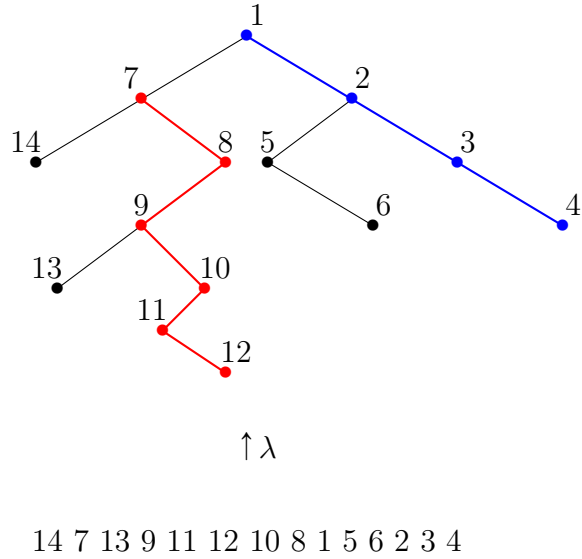
\begin{figure}
\centering
\begin{tikzpicture}[scale=0.28]

\node at  (0,0.5){$\uparrow$};\node at  (1,0.5){$\lambda$};
\node at  (0,-4){$14~ 7~ 13~ 9~ 11~ 12~ 10~ 8~ 1~ 5~ 6~ 2~ 3~ 4$};

\draw[-,blue,thick](0,20) to (15,11);
\draw[-](0,20) to (-10,14);
\draw[-,red,thick](-5,17) to (-1,14);
\draw[-](-9,8) to (-5,11);
\draw[-,red,thick](-5,11) to (-1,14);
\draw[-,red,thick](-5,11) to (-2,8);
\draw[-,red,thick](-4,6) to (-2,8);
\draw[-,red,thick](-4,6) to (-1,4);
\draw[-](5,17) to (1,14);
\draw[-](1,14) to (6,11);

\node at  (0,20){\blue{$\bullet$}};\node at  (0.5,21){$1$};
\node at  (-5,17){\red{$\bullet$}};\node at  (-5.5,18){$7$};
\node at  (-10,14){$\bullet$};\node at  (-10.5,15){$14$};
\node at  (-1,14){\red{$\bullet$}};\node at  (-1,15){$8$};
\node at  (-5,11){\red{$\bullet$}};\node at  (-5,12){$9$};
\node at  (-9,8){$\bullet$};\node at  (-9.5,9){$13$};
\node at  (-2,8){\red{$\bullet$}};\node at  (-1.5,9){$10$};
\node at  (-4,6){\red{$\bullet$}};\node at  (-4.5,7){$11$};
\node at  (-1,4){\red{$\bullet$}};\node at  (-0.5,5){$12$};
\node at  (5,17){\blue{$\bullet$}};\node at  (5.5,18){$2$};
\node at  (10,14){\blue{$\bullet$}};\node at  (10.5,15){$3$};
\node at  (15,11){\blue{$\bullet$}};\node at  (15.5,12){$4$};
\node at  (1,14){$\bullet$};\node at  (1,15){$5$};
\node at  (6,11){$\bullet$};\node at  (6.5,12){$6$};

\end{tikzpicture}
\caption{An example illustrating the bijection $\lambda$.\label{fig lambda}}
\end{figure}

There exists a natural injection from binary trees to increasing binary trees. The {\em natural labeling} of a binary tree in $\B_n$ is obtained by labeling its nodes with $[n]$ in {\em right pre-order}, which traverses the parent node first, then the right subtree, and finally the left subtree. See the labeling of the  binary tree in Fig.~\ref{fig lambda} for an example.  Under this natural injection, the mapping  $\lambda$ induces a bijection from  $\S_n(213)$ to $\B_n$.

For a binary tree $T$ (labeled or unlabeled), the {\em right arm} (resp.,~{\em left arm}) of $T$ is defined as the unique maximal path  starting from the root and containing only right (resp.,~left) edges.
A {\em  restricted  path} of $T$ is a path that does not involve any node in the right arm of $T$.
A {\em longest restricted right path} is a restricted path containing maximum number of right edges.  
For the tree $T$ in Fig.~\ref{fig lambda}, the right arm of $T$ is $1-2-3-4$ (in blue) and  the only longest restricted right path is $7-8-9-10-11-12$ (in red).
Let $\lrrp(T)$ denote one plus the number of right edges in any longest restricted right path of $T$. We use the convention that $\lrrp(T)=0$ for the special  tree $T$ with no left edges. 
The following lemma characterizes $t$-stack-sortable $213$-avoiding  permutations using the statistic $\lrrp$ on binary trees. 
\begin{lemma}\label{lem:213}
For $t\geq1$ and a $213$-avoiding permutation $\pi$, $\pi$  is $t$-stack-sortable if and only if $\lrrp(\lambda(\pi))\leq t$. 
\end{lemma}

Given a permutation $\pi\in\S_n$, let  $\Rlmi(\pi)$ be the set of all right-to-left minima of $\pi$. 
If $\pi\in\S_n(321)$, then introduce 
$$\mlw(\pi):=\max_{\pi_j\in\Rlmi(\pi)}\{\lw_j(\pi)\},$$
where $\lw_j(\pi):=j-\pi_j$. For example, if $\pi=3\,4\,1\,6\,8\,2\,5\,7\,9\,12\,10\,11\in\S_{12}(321)$, then $\mlw(\pi)=\lw_6(\pi)=4$. The next lemma characterizes $t$-stack-sortable $321$-avoiding  permutations using the statistic $\mlw$. 
\begin{lemma}\label{lem:321}
For $t\geq1$ and a $321$-avoiding permutation $\pi$, $\pi$ is $t$-stack-sortable if and only if $\mlw(\pi)\leq t$. 
\end{lemma}

The proofs of the above two characterizing  lemmas are technically  involved and will be postponed to  the next  section. 

\subsection{The construction of $\Upsilon$} 
With Lemmas~\ref{lem:213} and~\ref{lem:321} in hand, we can now construct $\Upsilon$. 
 Our bijection $\Upsilon$ is a composition of five bijections, four of which are known and one is an involution on binary trees. 
 
\magenta{{\bf Step 1:}} A {\em Dyck path} of order $n$ is a lattice path in $\N^2$ from $(0,0)$ to
$(n,n)$ using the {\em east step} $(1,0)$ and the {\em north step} $(0,1)$, which does
not pass above the diagonal $y=x$. Let $\D_n$ be the set of all Dyck paths of order $n$. 
Our first step of $\Upsilon$ is a bijection $\rho: \S_n(321)\rightarrow\D_n$ essentially due to Krattenthaler~\cite{Kra}.  It is known that a permutation is $321$-avoiding iff both the subsequence formed by  excedances  and the one formed by the remaining letters are increasing. For $\pi\in\S_n(321)$, we represent $\pi$ on the $n\times n$ grid with crosses on the squares $(i,\pi_i)$ for all $i$; see Fig.~\ref{fig tau} for the representation of $\pi=3\,4\,1\,6\,8\,2\,5\,7\,9\,12\,10\,11\in\S_{12}(321)$. With this representation,  $\rho(\pi)$ is the Dyck path from the lower-left corner to the upper-right corner of the grid leaving all the crosses to the left and remaining always as close to the diagonal as possible. See  Fig.~\ref{fig tau} (in left) for an example of $\rho$. 

\begin{figure}
    \centering
    
    \begin{tikzpicture}[scale=0.7]
\begin{scope}[scale=0.7]
\draw[help lines,step = 1] (0,0) grid (12,12);
\draw[] (0,0)--(12,12);
\draw[red] (2,0)--(3,1);
\draw[red] (2,1)--(3,0);

\draw[green] (5,1)--(6,2);
\draw[green] (5,2)--(6,1);

\draw[red] (6,4)--(7,5);
\draw[red] (6,5)--(7,4);

\draw[red] (7,6)--(8,7);
\draw[red] (7,7)--(8,6);

\draw[red] (8,8)--(9,9);
\draw[red] (8,9)--(9,8);

\draw[red] (10,9)--(11,10);
\draw[red] (10,10)--(11,9);

\draw[red] (11,10)--(12,11);
\draw[red] (11,11)--(12,10);

\draw[thick] (0,0)--(3,0)--(3,1)--(6,1)--(6,4)--(7,4)--(7,6)--(8,6)--(8,8)--(9,8)--(9,9)--(11,9)--(11,10)--(12,10)--(12,12);

\draw[blue] (0,2)--(1,3);
\draw[blue] (0,3)--(1,2);

\draw[blue] (1,3)--(2,4);
\draw[blue] (1,4)--(2,3);

\draw[blue] (3,5)--(4,6);
\draw[blue] (4,5)--(3,6);

\draw[blue] (4,7)--(5,8);
\draw[blue] (5,7)--(4,8);

\draw[blue] (9,11)--(10,12);
\draw[blue] (10,11)--(9,12);

\draw[green] (5,0)--(12,7);

\node at  (-0.8,-0.5){$\pi=$};
\node at  (0.5,-0.5){$3$};
\node at  (1.5,-0.5){$4$};
\node at  (2.5,-0.5){\red{$1$}};
\node at  (3.5,-0.5){$6$};
\node at  (4.5,-0.5){$8$};
\node at  (5.5,-0.5){\green{$2$}};
\node at  (6.5,-0.5){\red{$5$}};
\node at  (7.5,-0.5){\red{$7$}};
\node at  (8.5,-0.5){\red{$9$}};
\node at  (9.5,-0.5){$12$};
\node at  (10.5,-0.5){\red{$10$}};
\node at  (11.5,-0.5){\red{$11$}};
\end{scope}

\begin{scope}[xshift=19cm,yshift=0cm]

\draw[] (-5,7)--(-8,4);
\draw[] (-7,5)--(-7,2);
\draw[] (-7,5)--(-6,4);
\draw[] (-5,7)--(-4,6)--(-4,5);
\draw[] (-6,6)--(-6,5);
\draw[] (-5,7)--(-5,6);
\draw[] (-4,6)--(-3,5);
\draw[green] (-5,7)--(-7,5)--(-7,2);

\node at  (-3,5){\red{$\bullet$}};
\node at  (-5,6){\red{$\bullet$}};
\node at  (-6,5){\red{$\bullet$}};
\node at  (-5,7){$\bullet$};
\node at  (-6,6){$\bullet$};
\node at  (-7,5){$\bullet$};
\node at  (-8,4){\red{$\bullet$}};
\node at  (-7,4){$\bullet$};
\node at  (-7,3){$\bullet$};
\node at  (-7,2){\green{$\bullet$}};
\node at  (-6,4){\red{$\bullet$}};
\node at  (-4,6){{$\bullet$}};
\node at  (-4,5){\red{$\bullet$}};

\node at  (-5.8,4.6){$7$};
\node at  (-5,5.5){$9$};
\node at  (-3,4.5){$11$};
\node at  (-8,3.5){$1$};
\node at  (-7,1.5){$2$};
\node at  (-6,3.5){$5$};
\node at  (-4,4.5){$10$};

\node at  (-9,4.5) {$\rightarrow$};
\node at  (-9,4.8) {$\tau^{-1}$};

\end{scope}
\end{tikzpicture}
    \caption{An example of  $\tau^{-1}\circ\rho$.\label{fig tau}}
\end{figure}
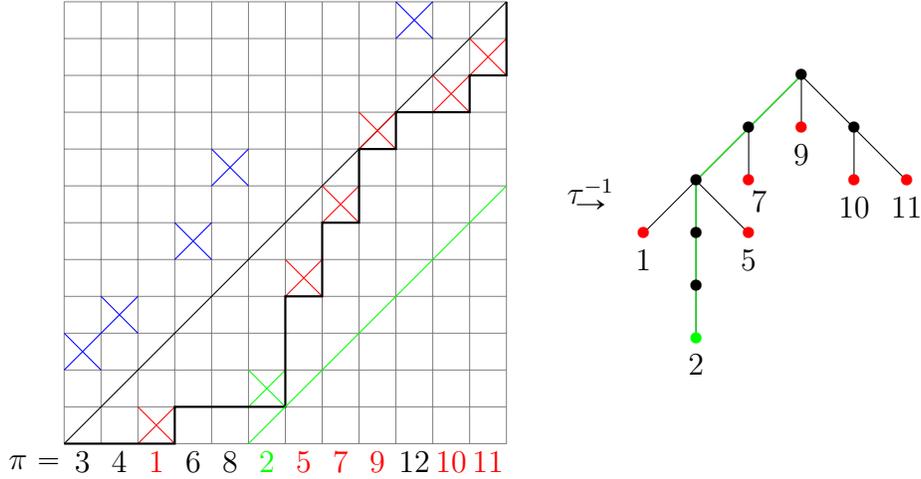

 Let $\D_{n,t}$ denote the set of Dyck paths in $\D_n$ that do not cross below the shifted diagonal $y=x-(t+1)$. By Lemma~\ref{lem:321}, $\rho$ restricts to a bijection between $\S_n^t(321)$ and $\D_{n,t}$.

 \magenta{{\bf Step 2:}} A {\em plane tree} is a rooted tree in which the children of each node are linearly ordered. Denote by $\T_n$ the set of all plane trees with $n$ edges. Our second step of $\Upsilon$ uses a classical  bijection $\tau$ (see~\cite{Deu}) between plane trees and Dyck paths, constructed as follows:
\begin{enumerate}[(i)]
    \item Traverse the tree in {\em pre-order}, that is, visiting the root first, then recursively the left subtree, followed by the right subtree.
    \item For each downward edge traversal, record an east step $(1,0)$.
    \item For each upward edge traversal, record a north step $(0,1)$.
\end{enumerate}
See Fig.~\ref{fig tau} for an example of $\tau$.

For a plane tree $T\in\T_n$, the {\em depth} of $T$ is the maximum number of edges in any root-to-leaf path in $T$. Let $\T_{n,t}$ be the set of all  trees in $\T_n$ with depth  at most $t+1$. Since each leaf of $T$ corresponds to a corner in $\tau(T)$, $\tau$ restricts to a bijection between $\T_{n,t}$ and $\D_{n,t}$. 

\magenta{{\bf Step 3:}} 
For a plane tree $T\in\T_n$,  two nodes with the same parent are called {\em siblings}. Our third step of $\Upsilon$ uses a classical bijection~\cite[p.~9]{St1} between plane trees $\T_n$ and binary trees $\B_n$. For a tree $T\in\T_n$, we define the binary tree $\varphi(T)\in\B_n$ by requiring that for each pair of non-root nodes $(x,y)$ in $T$: 
 \begin{enumerate}
    \item $y$ is the left child of $x$ in $\varphi(T)$ only if  when  $y$ is the leftmost child of  $x$ in $T$;
    \item  $y$ is the right child of $x$ in $\varphi(T)$ only if when $x$ is the closest left sibling of $y$ in $T$.
 \end{enumerate}
See Fig.~\ref{fig gamma} for an example of $\varphi$.

\begin{figure}
\centering
    
\begin{tikzpicture}[scale=0.7]

\node at  (-1.2,4.5){$\rightarrow$};
\node at  (-1.2,4.8){$\varphi$};

\node at  (6.8,4.5){$\rightarrow$};
\node at  (6.8,4.8){$\gamma$};

\draw[] (-5,7)--(-8,4);
\draw[] (-7,5)--(-7,2);
\draw[] (-7,5)--(-6,4);
\draw[] (-5,7)--(-3,5);
\draw[] (-5,7)--(-5,6);
\draw[] (-4,6)--(-4,5);
\draw[] (-6,6)--(-6,5);

\draw[green] (-5,7)--(-7,5)--(-7,2);

\node at  (-6,5){$\bullet$};
\node at  (-5,7){$\bullet$};
\node at  (-6,6){$\bullet$};
\node at  (-7,5){$\bullet$};
\node at  (-8,4){{$\bullet$}};
\node at  (-7,4){$\bullet$};
\node at  (-7,3){$\bullet$};
\node at  (-7,2){{$\bullet$}};
\node at  (-6,4){{$\bullet$}};
\node at  (-4,6){{$\bullet$}};
\node at  (-5,6){$\bullet$};
\node at  (-3,5){$\bullet$};
\node at  (-4,5){$\bullet$};


\begin{scope}[xshift=2.5cm,yshift=0cm]
\draw[] (0,7)--(2,5)--(1,4)--(2,3);
\draw[] (0,7)--(-2,5)--(0,3);
\draw[] (-1,4)--(-3,2);
\draw[] (-1,6)--(0,5);
\draw[green] (0,7)--(-2,5)--(-1,4)--(-3,2);

\node at (0,5){$\bullet$};
\node at (0,7){$\bullet$};
\node at (-1,6){\magenta{$\bullet$}};
\node at (-2,5){$\bullet$};
\node at (-1,4){$\bullet$};
\node at (0,3){$\bullet$};
\node at (-2,3){$\bullet$};
\node at (-3,2){$\bullet$};
\node at (1,6){$\bullet$};
\node at (2,5){$\bullet$};
\node at (1,4){\magenta{$\bullet$}};
\node at (2,3){$\bullet$};
\end{scope}


\begin{scope}[xshift=10cm,yshift=0cm]
\draw[] (0,7)--(2,5)--(0.5,3.5);
\draw[] (0,7)--(-1,6)--(0,5)--(-2,3);
\draw[] (-1,4)--(1,2);
\draw[] (-1,6)--(-2,5);
\draw[green] (-1,6)--(0,5)--(-1,4)--(1,2);

\node at (-2,5){$\bullet$};
\node at (0,7){$\bullet$};
\node at (-1,6){\magenta{$\bullet$}};
\node at (0,5){$\bullet$};
\node at (-1,4){$\bullet$};
\node at (0,3){$\bullet$};
\node at (-2,3){$\bullet$};
\node at (1,2){$\bullet$};
\node at (1,6){$\bullet$};
\node at (2,5){$\bullet$};
\node at (1.25,4.25){\magenta{$\bullet$}};
\node at (0.5,3.5){$\bullet$};
\end{scope}

\end{tikzpicture}
    \caption{An example of $\gamma\circ\varphi$.\label{fig gamma}}
\end{figure}
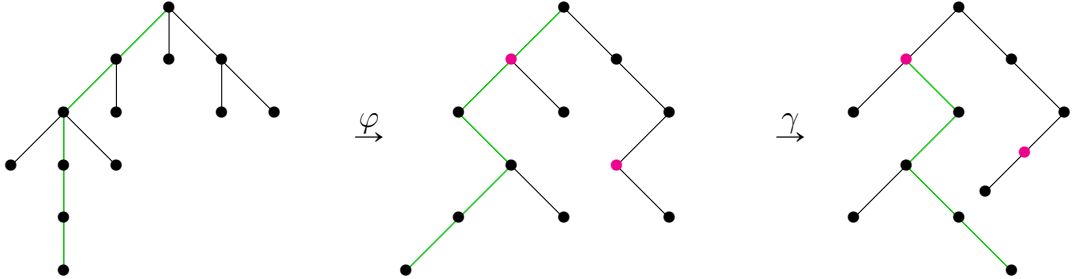

For a binary tree $T\in\B_n$, a {\em longest left path}  of $T$ is a path   that contains the maximum number of left edges. We denote  by $\llp(T)$ the number of left edges in any longest left path of $T$. For the second tree $T$ in Fig.~\ref{fig gamma}, we have $\llp(T)=4$. Let 
$$
\B_{n,t}:=\{T\in\B_n: \llp(T)\leq t\}. 
$$
It is plain to see from the construction that $\varphi$ restricts to a bijection between $\T_{n,t}$ and $\B_{n,t}$. 

\magenta{\bf Step 4:} 
Let $\beta$ denote the {\em mirror symmetry} operation on binary trees, namely if $T$ is a binary tree, then $\beta(T)$ is the mirror symmetry of $T$.  We now define a {\em restricted mirror symmetry} operation $\gamma$ on $\B_n$ that exchanges the  pair of tree statistics $(\llp,\lrrp)$. Given a binary tree $T\in\B_n$, $\gamma(T)$ is obtained by applying mirror symmetry to every subtree whose root is a left child of a node on the right arm of $T$. See Fig.~\ref{fig gamma} for an example of $\gamma$. It is clear that $\gamma$ is an involution on $\B_n$ that exchanges the statistic pair  $(\llp,\lrrp)$. In particular, $\gamma$ restricts to a bijection between $\B_{n,t}$ and 
$$
\widetilde\B_{n,t}:=\{T\in\B_n: \lrrp(T)\leq t\},
$$
 which forms the fourth step of our $\Upsilon$.

\magenta{\bf Step 5:} The step 5 of our $\Upsilon$ uses the bijection $\lambda^{-1}:\B_n\rightarrow\S_n(213)$. Note that when applying $\lambda^{-1}$, each binary tree is assumed to be equipped with labeling using right pre-order (see Fig.~\ref{fig lambda}). By Lemma~\ref{lem:213}, the mapping $\lambda^{-1}$ establishes a bijection between $\widetilde\B_{n,t}$ and $\S_n^t(213)$. 
\vskip0.1in
Now setting 
$$
\Upsilon= \lambda^{-1}\circ\gamma\circ\varphi\circ\tau^{-1}\circ\rho
$$
gives the desired bijection for Theorem~\ref{main:bij}. 

\subsection{The transformation of permutation statistics under $\Upsilon$}
\label{subsection:2.2}
In this subsection, we investigate the transformation of permutation statistics under $\Upsilon$. 
First we prove Theorem~\ref{main:bij}.

\begin{proof}[{\bf Proof of Theorem~\ref{main:bij}}]
It remains  to verify that $\Upsilon$ transforms the pair  $(\fix,\drop)$ on $\S_n^t(321)$ to $(\bad,\des)$ on $\S_n^t(213)$. 

Let $\pi\in\S_n^t(321)$ and $\sigma=\Upsilon(\pi)\in\S_n^t(213)$. Then it is straightforward to verify  that $\fix(\pi)$ equals the number of corners on the diagonal $y=x$ in the Dyck path $\rho(\pi)$, which becomes the number of leaves whose parent is the root in the plane tree  $\tau^{-1}\circ\rho(\pi)$. Then under $\varphi$, this quantity turns to the number of nodes on the right arm that has no left child in the binary tree $\varphi\circ\tau^{-1}\circ\rho(\pi)$, which is preserved under the involution $\gamma$. Finally, it becomes $\bad(\sigma)$ under $\lambda^{-1}$, as desired. The proof that $\drop(\pi)$ turns to $\des(\sigma)$ is similar and will be omitted. This completes the proof of Theorem~\ref{main:bij}. 
\end{proof}

Since our bijection $\Upsilon$ is natural, it transforms  more permutation statistics between $\S_n^t(321)$  and $\S_n^t(213)$ that we now introduce. For the $321$-avoiding permutations side, define 
\begin{itemize}
\item $\pone(\pi)$, the position of the letter $1$ in $\pi$;
\item $\prmi(\pi)$, the number of prime right-to-left minima of $\pi$, where a right-to-left minimum $\pi_{i_k}\in\Rlmi(\pi)=\{\pi_{i_1}<\pi_{i_2}<\cdots<\pi_{i_m}\}$ is {\em prime} if $k=1$ or it satisfies $\pi_{i_k}-1=i_{k-1}$ when $k\geq2$;
\item $\inv(\pi):=|\{(i,j)\in[n]^2: i<j, \pi_i>\pi_j\}|$, the number of {\em inversions} of $\pi$.
\end{itemize}
For example, if $\pi=3\,4\,\red{\bf1}\,6\,8\,\red{\bf2\,5\,7\,9}\,12\,\red{\bf10\,11}$, then $\pone(\pi)=3$, $\prmi(\pi)=3$ and $\inv(\pi)=11$. For the $213$-avoiding permutations side, define 
\begin{itemize}
\item $\rrlmi(\pi)$, the number of restricted right-to-left minima of $\pi$, where a letter $\pi_i$ with $i\leq\pone(\pi)$ is a  {\em restricted right-to-left minimum} if  $\pi_i<\pi_k$ for all $i<k<\pone(\pi)$;
\item $\rlmi(\pi)$, the number of  right-to-left minima of $\pi$;
\item $\rcinv(\pi)$, the number of  {\em restricted coinversions} of $\pi$, where a pair $(i,j)$ with  $\pi_i\leq\pi_j$ is a {\em restricted coinversion} if $\pi_j$ is a descent top of $\pi$ and  $i_{k-1}<i\leq j<i_k$ for some $k$, assuming $\Rlmi(\pi)=\{\pi_{i_1}<\pi_{i_2}<\cdots<\pi_{i_m}\}$ and $i_0=0$.
\end{itemize}
For example, if $\pi=12\,6\,11\,8\,9\,10\,7\,\red{{\bf1}}\,\red{{\bf2}}\,5\,4\,\red{{\bf3}}$, then $\rrlmi(\pi)=3$, $\rlmi(\pi)=3$ and $\rcinv(\pi)=11$.

\begin{proposition}
The bijection $\Upsilon$ transforms the triple $(\pone,\prmi,\inv)$ on $\S_n^t(321)$ to the triple 
$(\rrlmi,\rlmi,\rcinv)$ on $\S_n^t(213)$. 
\end{proposition}
\begin{proof}

We only show how $\Upsilon$ transforms $\inv$ on $\S_n^t(321)$  to $\rcinv$ on $\S_n^t(213)$, as the other transformations of statistics are evident from the construction of $\Upsilon$.

Let $\pi\in\S_n(321)$. For each $\pi_j\in\Rlmi(\pi)$, as $\pi$ is $321$-avoiding, the number of inversions of the form  $(i,j)$ with $i<j$ equals $j-\pi_j=\lw_j(\pi)$. On the other hand, if $\pi_j\notin\Rlmi(\pi)$, then all letters to the left of $\pi_j$ are smaller than $\pi_j$. Thus, 
\begin{equation}\label{eq:rcinv}
\inv(\pi)=\sum_{\pi_j\in\Rlmi(\pi)}\lw_j(\pi)
\end{equation}
for $\pi\in\S_n(321)$.

    For any $\pi_j\in\Rlmi(\pi)$ with $j\neq\pi_j$, let $v$ be its corresponding node in $\varphi\circ\tau^{-1}\circ\rho(\pi)=T$. It is straightforward to check that $\lw_j(\pi)$ equals  the number of left edges in the  path from $v$ to the root in $T$. Under the involution $\gamma$, $\lw_j(\pi)-1$ turns to the number of right edges from $v$ to the right arm (terminates at the node $w$ on the right arm) in $\gamma(T)$ and $v$ becomes a right leaf (see Def.~\ref{ri:leaf}) of $\gamma(T)$.  
    Suppose that  $\Upsilon(\pi)=\sigma$ with $\Rlmi(\sigma)=\{\sigma_{i_1}<\sigma_{i_2}<\cdots<\sigma_{i_m}\}$   and the natural labelings of $v$ and $w$ 
    are, respectively,  $\sigma_t$ and $\sigma_{i_k}$. Note that $\sigma_t$ is a descent top of $\sigma$. 
 Finally,  under the bijection $\lambda^{-1}$, $\lw_j(\pi)$ becomes $|\{(s,t): \sigma_s<\sigma_t, i_{k-1}<s\leq t<i_k\}|$, where we use the convention $i_0=0$. This completes the proof that $\inv(\pi)=\rcinv(\sigma)$ in view of~\eqref{eq:rcinv}.  
\end{proof}

\section{Proofs of the two characterizing lemmas}\label{sec:3}

This section aims to prove the two characterizing lemmas, Lemmas~\ref{lem:213} and~\ref{lem:321}.

\subsection{Proof of Lemma~\ref{lem:213}}
In this subsection, we characterize the stack-sorting algorithm for permutations in terms of increasing binary trees and prove  Lemma~\ref{lem:213}.

\begin{definition}\label{ri:leaf}
For a binary tree $T\in\IB_n$, a node is called a {\bf right leaf} if it has no right child and is not belonging to the right arm of $T$.
\end{definition}

\begin{definition}
    For a binary tree $T\in\IB_n$ and a right leaf $v\in T$, the operation $\phi_v$ transforms $T$ by inserting $v$ to the right arm while preserving the increasing property, and then replacing the original position of $v$ with the left child of $v$ (if exists). See Fig.~\ref{fig phi} for an example of the operation $\phi_v$. 
\end{definition}

\begin{figure}
\centering
\begin{tikzpicture}[scale=0.235]
\draw[-](0,20) to (15,11);
\draw[-](0,20) to (-10,14);
\draw[-](-5,17) to (-1,14);
\draw[-](-9,8) to (-1,14);
\draw[-](-5,11) to (-2,8);
\draw[-](-4,6) to (-2,8);
\draw[-](-4,6) to (-1,4);
\draw[-](5,17) to (1,14);
\draw[-](1,14) to (6,11);

\node at  (0,20){$\bullet$};\node at  (0.5,21.2){$1$};
\node at  (-5,17){$\bullet$};\node at  (-5.5,18.2){$5$};
\node at  (-10,14){{$\bullet$}};\node at  (-10.5,15.2){$14$};
\node at  (-1,14){\red{$\bullet$}};\node at  (-1,15.2){$6$};
\node at  (-5,11){\blue{$\bullet$}};\node at  (-5,12.2){$7$};
\node at  (-9,8){{$\bullet$}};\node at  (-9.5,9.2){$13$};
\node at  (-2,8){{$\bullet$}};\node at  (-1.5,9.2){$9$};
\node at  (-4,6){$\bullet$};\node at  (-4.5,7.2){$10$};
\node at  (-1,4){{$\bullet$}};\node at  (-0.5,5.2){$12$};
\node at  (5,17){{$\bullet$}};\node at  (5.5,18.2){$2$};
\node at  (10,14){{$\bullet$}};\node at  (10.5,15.2){$8$};
\node at  (15,11){{$\bullet$}};\node at  (15.5,12.2){$11$};
\node at  (1,14){$\bullet$};\node at  (1,15.2){$3$};
\node at  (6,11){{$\bullet$}};\node at  (6.5,12.2){$4$};

\begin{scope}[xshift=35cm,yshift=0cm]
    \draw[-](0,20) to (15,11);
\draw[-](0,20) to (-10,14);
\draw[-](-5,17) to (-1,14);
\draw[-](15,11) to (20,8);
\draw[-](5,17) to (2,14);
\draw[-](2,14) to (7,11);

\node at  (0,20){$\bullet$};\node at  (0.5,21.2){$1$};
\node at  (-5,17){$\bullet$};\node at  (-5.5,18.2){$5$};
\node at  (-10,14){{$\bullet$}};\node at  (-10.5,15.2){$14$};
\node at  (20,8){{$\bullet$}};\node at  (20.5,9.2){$11$};

\node at  (5,17){{$\bullet$}};\node at  (5.5,18.2){$2$};
\node at  (10,14){\red{$\bullet$}};\node at  (10.5,15.2){$6$};
\node at  (15,11){{$\bullet$}};\node at  (15.5,12.2){$8$};
\node at  (2,14){$\bullet$};\node at  (2,15.2){$3$};
\node at  (7,11){{$\bullet$}};\node at  (7.5,12.2){$4$};
\end{scope}

\begin{scope}[xshift=39cm,yshift=3cm]
    \draw[-](-5,11) to (-2,8);
    \draw[-](-9,8) to (-5,11);
    \draw[-](-4,6) to (-2,8);
    \draw[-](-4,6) to (-1,4);

    \node at  (-5,11){\blue{$\bullet$}};\node at  (-5,12.2){$7$};
    \node at  (-2,8){{$\bullet$}};\node at  (-1.5,9.2){$9$};
    \node at  (-4,6){$\bullet$};\node at  (-4.5,7.2){$10$};
    \node at  (-1,4){{$\bullet$}};\node at  (-0.5,5.2){$12$};
    \node at  (-9,8){{$\bullet$}};\node at  (-9.5,9.2){$13$};
\end{scope}
\node at  (20,11){$\rightarrow$};\node at  (20,12.5){$\phi_6$};
\end{tikzpicture}
\caption{An example of the operation $\phi_6$.\label{fig phi}}
\end{figure}
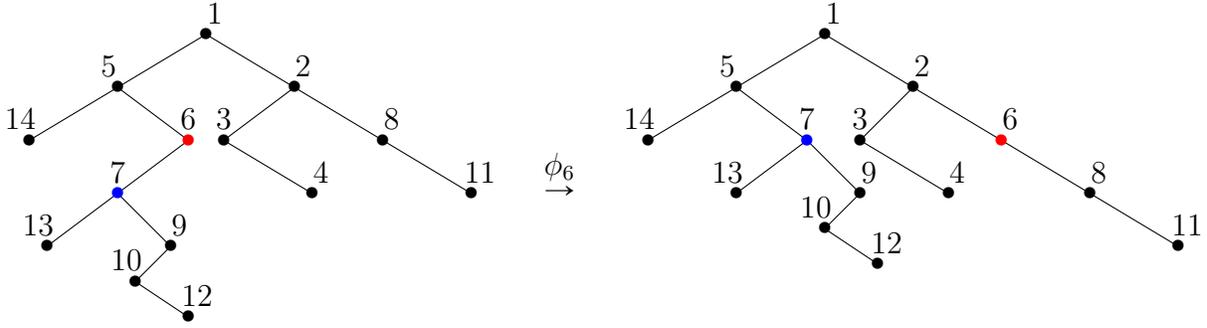

For a binary tree $T\in\IB_n$, the operation $\phi$ is commutative, namely, for any two right leaves $u$ and $v\in T$, we have $\phi_v\circ\phi_u(T)=\phi_u\circ\phi_v(T)$. Let $\RL(T)$ denote the set of all right leaves of $T$. We can define the sorting algorithm on $T$ by 
$$\Phi(T)=\biggl(\prod_{v\in \RL(T)}\phi_v\biggr)(T).$$
See Fig.~\ref{fig:Phi} for an example of the sorting algorithm $\Phi$.
The operator $\Phi$ satisfies the following property.

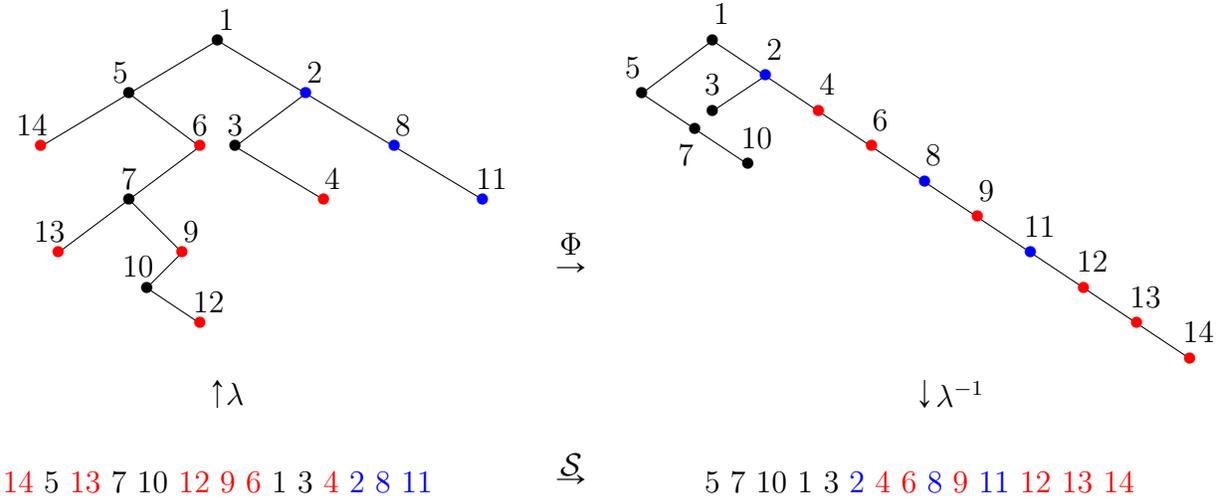
\begin{figure}
\centering
\begin{tikzpicture}[scale=0.235]
\node at  (0,0){$\uparrow$};\node at  (1,0){$\lambda$};
\node at  (0,-5){$\red{14}~ 5~ \red{13}~ 7~ 10~ \red{12~ 9~ 6~} 1~ 3~ \red{4}~ \blue{2~ 8~ 11}$};

\draw[-](0,20) to (15,11);
\draw[-](0,20) to (-10,14);
\draw[-](-5,17) to (-1,14);
\draw[-](-9,8) to (-1,14);
\draw[-](-5,11) to (-2,8);
\draw[-](-4,6) to (-2,8);
\draw[-](-4,6) to (-1,4);
\draw[-](5,17) to (1,14);
\draw[-](1,14) to (6,11);

\node at  (0,20){$\bullet$};\node at  (0.5,21.2){$1$};
\node at  (-5,17){$\bullet$};\node at  (-5.5,18.2){$5$};
\node at  (-10,14){\red{$\bullet$}};\node at  (-10.5,15.2){$14$};
\node at  (-1,14){\red{$\bullet$}};\node at  (-1,15.2){$6$};
\node at  (-5,11){$\bullet$};\node at  (-5,12.2){$7$};
\node at  (-9,8){\red{$\bullet$}};\node at  (-9.5,9.2){$13$};
\node at  (-2,8){\red{$\bullet$}};\node at  (-1.5,9.2){$9$};
\node at  (-4,6){$\bullet$};\node at  (-4.5,7.2){$10$};
\node at  (-1,4){\red{$\bullet$}};\node at  (-0.5,5.2){$12$};
\node at  (5,17){\blue{$\bullet$}};\node at  (5.5,18.2){$2$};
\node at  (10,14){\blue{$\bullet$}};\node at  (10.5,15.2){$8$};
\node at  (15,11){\blue{$\bullet$}};\node at  (15.5,12.2){$11$};
\node at  (1,14){$\bullet$};\node at  (1,15.2){$3$};
\node at  (6,11){\red{$\bullet$}};\node at  (6.5,12.2){$4$};

\node at  (20,-5){$\rightarrow$};\node at  (20,-4){$\s$};
\node at  (40,-5){$5~ 7~ 10~ 1~ 3~ \blue{2~} \red{4~ 6~} \blue{8~} \red{9~} \blue{11~} \red{12~ 13~ 14~}$};
\node at  (40,0){$\downarrow$};\node at  (42,0){$\lambda^{-1}$};
\node at  (20,7){$\rightarrow$};\node at  (20,8.5){$\Phi$};

\draw[-](28,20) to (55,2);
\draw[-](28,20) to (24,17);
\draw[-](24,17) to (30,13);
\draw[-](31,18) to (28,16);

\node at  (28,20){$\bullet$};\node at  (28.5,21.5){$1$};
\node at  (31,18){\blue{$\bullet$}};\node at  (31.5,19.5){$2$};
\node at  (34,16){\red{$\bullet$}};\node at  (34.5,17.5){$4$};
\node at  (37,14){\red{$\bullet$}};\node at  (37.5,15.5){$6$};
\node at  (40,12){\blue{$\bullet$}};\node at  (40.5,13.5){$8$};
\node at  (43,10){\red{$\bullet$}};\node at  (43.5,11.5){$9$};
\node at  (46,8){\blue{$\bullet$}};\node at  (46.5,9.5){$11$};
\node at  (49,6){\red{$\bullet$}};\node at  (49.5,7.5){$12$};
\node at  (52,4){\red{$\bullet$}};\node at  (52.5,5.5){$13$};
\node at  (55,2){\red{$\bullet$}};\node at  (55.5,3.5){$14$};
\node at  (24,17){$\bullet$};\node at  (23.5,18.5){$5$};
\node at  (27,15){$\bullet$};\node at  (26.5,13.5){$7$};
\node at  (30,13){$\bullet$};\node at  (30.5,14.5){$10$};
\node at  (28,16){$\bullet$};\node at  (28,17.5){$3$};

\end{tikzpicture}
\caption{An example of the sorting algorithm $\Phi$.\label{fig:Phi}}
\end{figure}

\begin{observation}\label{prop phi}
     For any $T\in \IB_n$ with $\lrrp(T)\geq1$, we have $\lrrp(\Phi(T))=\lrrp(T)-1$.
\end{observation}
\begin{proof}
    The difference between $\lrrp(T)$ and $\lrrp(\Phi(T))$ is caused by the removal of the largest node in the longest restricted right path of $T$.
\end{proof}

For our purpose, we need to introduce  a generalization of   $213$-avoiding permutations. For a permutation $\pi\in\S_n$, we define its {\em tail} as the longest consecutive increasing subsequence ending at $\pi_n$. For example, the tail of the permutation $65129\red{3478}$ is $3478$. We call $\pi$ a {\em tailed $213$-avoiding permutation} if the subsequence obtained by removing its tail, except the smallest element, avoids the pattern $213$.
We denote by $\widetilde{\S}_n(213)$ the set of all  tailed $213$-avoiding permutations in $\S_n$. See Fig.~\ref{fig:Phi} (on the left) for an example of a tailed $213$-avoiding permutation that is not  $213$-avoiding.

Similarly, for a tree $T\in\IB_n$, we define its {\em tail} as the path starting from the largest node in the right arm of $T$ and tracing back to the first node that has left child (if no such node exists, the entire tree is considered as the tail). We call $T$ a {\em tailed binary tree} if, after removing its tail, except the smallest node, and relabeling the remaining nodes by replacing the $i$-th smallest  node with label $i$, the resulting tree is a naturally labeled binary tree. The set of all  tailed binary trees with $n$ nodes is denoted by $\widetilde{\B}_n$. See Fig.~\ref{fig:Phi} (on the left) for an example of a tailed binary tree that is not naturally labeled.

We write $\Tail(\pi)$ (resp.,~$\Tail(T)$) for the tail of the permutation $\pi$ (resp.,~tree $T$).
For $\pi\in\S_n$,  let $\Rlma(\pi):=\{\pi_i: \pi_i>\pi_j\text{ for all $j>i$}\}$ be  the set of all {\em right-to-left maxima } of $\pi$.  The following observation is clear from the construction of $\lambda$.

\begin{observation}\label{ob lambda}
The mapping $\lambda$ restricts  to a bijection between $\widetilde{\S}_n(213)$  and  $\widetilde{\B}_n$. Moreover, if we write $\pi\in \widetilde{\S}_n(213)$ as $\pi=\sigma\cdot\Tail(\pi)$, then
$$\Rlma(\sigma)=\RL(\lambda(\pi)) \quad\text{and} \quad \Tail(\pi)=\Tail(\lambda(\pi)).$$
\end{observation}

Based on  Observation~\ref{ob lambda}, we can translate  the sorting algorithm $\s$ on tailed $213$-avoiding permutations to the operation $\Phi$ on tailed binary trees (see Fig.~\ref{fig:Phi} for an example), as formalized in the following theorem.

\begin{lemma}\label{lemma sort}
    For any $\pi\in\widetilde{\S}_n(213)$, we have $\s(\pi)=\lambda^{-1}\circ\Phi\circ\lambda(\pi)$ and $\s(\pi)\in\widetilde{\S}_n(213)$. 
\end{lemma}
\begin{proof}
    Given $\pi\in\widetilde{\S}_n(213)$ of  the form $\pi=\sigma\cdot\Tail(\pi)$, we can decompose $\pi$ as 
    $$\pi=I_1\sigma_{i_1} I_2\sigma_{i_2}\cdots I_m\sigma_{i_m}\cdot\Tail(\pi),$$
    where $\Rlma(\sigma)=\{\sigma_{i_1}>\sigma_{i_2}>\cdots>\sigma_{i_m}\}$ and each $I_j$ is an increasing subsequence whose letters are smaller than $\sigma_{i_j}$. 
    Based on this decomposition and the recursive definition of the stack-sorting operation $\s$, we see that  $\s(\pi)$ can be obtained as follows:
    \begin{enumerate}[(i)]
    \item All elements in $\Rlma(\sigma)$ are inserted into $\Tail(\pi)$ to form a new increasing sequence.
    \item The remaining part of $\sigma$ stays unchanged, with the new increasing sequence appended at the end.
    \end{enumerate}
    The resulting permutation $\s(\pi)$ is still a  tailed $213$-avoiding permutation. For example, if $\pi=\red{14}\,5\,\red{13}\,7\,10\,\red{12\,9\,6}\,1\,3\,\red{4}\,\blue{2\,8\,11}$, then $\s(\pi)=5\,7\,10\,1\,3\,\blue{2}\,\red{4\,6}\,\blue{8}\,\red{9}\,\blue{11}\,\red{12\,13\,14}$. 

    The operation $\Phi$ on $\lambda(\pi)$ moves exactly the nodes in $\RL(\lambda(\pi))$ to the right arm while preserving the projecting positions of the other nodes. Since $\lambda(\pi)$ is a tailed binary tree, the nodes in $\RL(\lambda(\pi))$ are moved into the tail of  $\lambda(\pi)$. The result then follows from Observation~\ref{ob lambda}.
\end{proof}

\begin{remark}
Note that $\pi\in\S_n(213)$ does not guarantee that $\s(\pi)\in\S_n(213)$. Thus, the introduction of the extension $\widetilde\S_n(213)$ is necessary in our approach. 
\end{remark}

We are ready for the proof of  Lemma~\ref{lem:213}.
\begin{proof}[{\bf Proof of  Lemma~\ref{lem:213}}]
By Observation~\ref{prop phi} and Lemma~\ref{lemma sort}, we conclude that $\pi\in\widetilde{\S}_n(213)$ is $t$-stack-sortable iff $\lrrp(\lambda(\pi))\leq t$. Lemma~\ref{lem:213} then follows from the fact that $\S_n(213)\subseteq\widetilde{\S}_n(213)$. 
\end{proof}

\subsection{Proof of Lemma~\ref{lem:321}}

This subsection deals with Lemma~\ref{lem:321}. 
First, we describe the images, named $321$-avoiding trees, of the  $321$-avoiding permutations under the bijection $\lambda$.

Given an increasing binary tree $T$, a {\em right chain} in $T$  is any maximal path (except the right arm) composed of only right edges. 
If a right chain whose leading node is the left child of $v$, then this right chain is denoted by $C_v$. For instance, $C_3$ is the right chain $7-8-9$ in Fig.~\ref{fig:321-tree}. 
For any two right chains $C_v$ and $C_w$ in $T$, we write $C_v<C_w$ if every label in $C_v$ is smaller than every label in $C_w$.

\begin{definition}[$321$-avoiding trees]
An increasing binary tree $T\in\IB_n$ is called {\em a $321$-avoiding tree} if it satisfies
\begin{enumerate}
\item the path from the root to any right leaf  contains exactly  one left edge (equivalently, every right chain is attached to a node in the right arm);

\item for any two right chains $C_v$ and $C_w$ in $T$, $C_v<C_w$ iff $v<w$.
\end{enumerate}
See Fig.~\ref{fig:321-tree} for an example of $321$-avoiding tree. 
\end{definition}
Denote by $\LB_n$ the set of all $321$-avoiding trees in $\IB_n$.  We shall see that $\lambda$ restricts to a bijection between $\S_n(321)$ and $\LB_n$.

\begin{figure}
\centering
\begin{tikzpicture}[scale=0.18]
\draw[-](0,20) to (30,2);
\draw[-](0,20) to (-10,15);
\draw[-](5,17) to (1,14);
\draw[-](20,8) to (15,5);
\draw[-](25,5) to (20,2);
\draw[-](23,-2) to (20,2);
\draw[-](-10,15) to (-4,5);
\draw[-](1,14) to (7,4);

\node at  (0,20){$\bullet$};\node at  (1,21){$1$};
\node at  (5,17){$\bullet$};\node at  (6,18){$3$};
\node at  (10,14){$\bullet$};\node at  (11,15){$6$};
\node at  (15,11){$\bullet$};\node at  (16,12.5){$10$};
\node at  (20,8){$\bullet$};\node at  (21,9.5){$11$};
\node at  (25,5){$\bullet$};\node at  (26,6.5){$13$};
\node at  (30,2){$\bullet$};\node at  (31,3.5){$15$};
\node at  (-10,15){$\bullet$};\node at  (-11,16){$2$};
\node at  (-7,10){$\bullet$};\node at  (-8,9){$4$};
\node at  (-4,5){\red{$\bullet$}};\node at  (-5,4){$5$};
\node at  (1,14){$\bullet$};\node at  (0,13){$7$};
\node at  (4,9){$\bullet$};\node at  (3,8){$8$};
\node at  (7,4){\red{$\bullet$}};\node at  (6,3){$9$};
\node at  (15,5){\red{$\bullet$}};\node at  (13.6,4){$12$};
\node at  (23,-2){\red{$\bullet$}};\node at  (24.5,-1.7){$16$};
\node at  (20,2){$\bullet$};\node at  (18.6,1){$14$};
\end{tikzpicture}
\caption{A $321$-avoiding tree in $\LB_{16}$.\label{fig:321-tree}}
\end{figure}
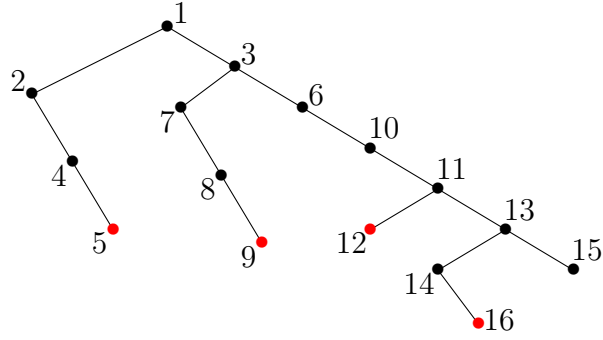

\begin{remark}
For any $T\in \LB_n$, a node $v$ is a right leaf in $T$ iff $v$ is the last element in some right chain in $T$.
\end{remark}

Any permutation $\pi\in\S_n(321)$ admits a unique decomposition as
\begin{equation}\label{eq:cannon}
\pi=I_1 a_1 I_2 a_2 I_3 a_3\cdots I_l a_l,
\end{equation}
where 
\begin{itemize}
\item[(1)] $\Rlmi(\pi)=\{a_1<a_2<\cdots<a_l\}$;
\item[(2)] the concatenation $I_1I_2\cdots I_l$ is an increasing sequence, where $I_j$ can be empty.
\end{itemize}
This decomposition is called the {\em canonical decomposition} of $321$-avoiding permutations. For instance, $\pi=245\red{1}789\red{3}\red{6}\in\S_{9}(321)$, where the non-right-to-left minima are increasing from left to right.
Based on the canonical decomposition of $321$-avoiding permutations and the definition of the map $\lambda$, the following observation on $\lambda$ can be checked routinely.

\begin{observation} \label{lem:bijpro}
The map $\lambda$ is a bijection from $\S_n(321)$ to $\LB_n$. More precisely, for $\pi\in \S_n(321)$ with canonical decomposition~\eqref{eq:cannon} and $T=\lambda(\pi)$, we have
\begin{itemize}
\item $\pi_i$ is a right-to-left minimum of $\pi$ iff $\pi_i$ is on the right arm of $T$;
\item $\pi_i$ and $\pi_j$ belong to the same part $I_k$ iff $\pi_i$ and $\pi_j$ are located at the same right chain of $T$;
\item  $\pi_i \in \DES(\pi)$  iff $\pi_i$ is a right leaf in $T$, where $\DES(\pi)$ is the set of descent tops of $\pi$.
\end{itemize}
\end{observation}

Next, we  turn to characterizing the stack-sorting map $\s$ on $321$-avoiding trees under $\lambda$.
In addition to its recursive definition given in \eqref{def:stack rec}, the map $\s$ was originally defined by West \cite{West} as the output of passing $\pi$ through a single stack. Bousquet-M\'elou \cite{Bous} further observed that the image $\s(\pi)$ coincides with the post-order reading word $P(T)$, where $T$ is the {\em decreasing binary tree} whose in-order reading word is $\pi$.
For our purpose,   we present a {\em dynamic description} of the stack-sorting map $\s$ on permutations.


For any permutation $\pi\in\S_n$, the output permutation $\s'(\pi)$ is obtained by relocating each descent top $\pi_i\in\DES(\pi )$ in a specified order. 
For each element $\pi_i\in\DES(\pi)$, let $\rfp(\pi_i)$ be the right-first element larger than $\pi_i$ (if it exists).
Each relocation corresponding to the descent top $\pi_i$, denoted by $\s_{\pi_i}$, is carried out according to the following rules:

\begin{itemize}
\item[(1)] place $\pi_i$ immediately before $\rfp(\pi_i)$, if such an entry exists;
\item[(2)] if no such entry exists (i.e., $\pi_i$ is a right-to-left maximum), then place $\pi_i$ at the end of the output.
\end{itemize}
Set $\s'(\pi)=\s_{b_k}\circ\s_{b_{k-1}}\circ\cdots\circ\s_{b_1}(\pi)$, where $\DES(\pi)=\{b_1<b_2<\cdots<b_k\}$.

\begin{example} 
Let $\pi= \red{8}45\red{962}1\red{7}3$, we have 
$\s_{2}(\pi)=\red{8}45\red{96}12\red{7}3$,
$\s_{6}\circ\s_{2}(\pi)=\red{8}45\red{9}126\red{7}3$,
$\s_{7}\circ\s_{6}\circ \s_2(\pi)=\red{8}45\red{9}12637$,
$\s_{8}\circ\s_{7}\circ \s_6\circ\s_{2}(\pi)=458\red{9}12637$
and $\s'(\pi)=458126379$.
\end{example}

\begin{proposition}\label{def:dyn}
For any $\pi\in\S_n$, the output permutation $\s'(\pi)$ coincides with $\s(\pi)$.
\end{proposition}

\begin{proof}
We write $\pi$ as $\pi=\pi_Ln\pi_R$, where the part $\pi_L$ (resp.,~$\pi_R$) is the subword before (resp.,~after) the element $n$. We distinguish the following two cases. 
\begin{itemize}
\item If $\pi_R=\emptyset$, then $\s'(\pi)=\s'(\pi_L)n$.
\item Otherwise, we have $b_k=n$.
If a descent top $b_i$ ($i\neq k$) occurs before (resp.,~after) $n$ in $\pi$,  then it remains before (resp.,~after) $n$ in the permutation $\s_{b_{k-1}}\circ\cdots\circ\s_{b_1}(\pi)$. This property ensures that for any $b_i$ occurring before $n$ and any $b_j$ occurring after $n$, the operations $\s_{b_i}$ and $\s_{b_j}$ are commutative, i.e., $\s_{b_i}\circ\s_{b_j}=\s_{b_j}\circ\s_{b_i}$.
As a result, we obtain the recursive decomposition $\s'(\pi)=\s'(\pi_L)\s'(\pi_R)n$.
\end{itemize}
Both cases satisfy the same recursion as $\s(\pi)$, which completes the proof. 
\end{proof}

For a $321$-avoiding tree $T\in \LB_n$, denote by $\RA(T)$ the set of all nodes in the right arm of $T$. For $u\in\RA(T)$, let $T_r(u)$ denote the right subtree of the node $u$.   Given a node $v\in\RL(T)$,  suppose that $v$ is the last node of the right chain $C_p$ in $T$ for some $p\in\RA(T)$.
Let $\rft(v)$ be the first node under {\em in-order} in $T_r(p)$ that is greater than $v$ (if it exists) in $T$.
The action $\theta_v(T)$ is defined  according to the following three cases:
\begin{itemize}
\item[(a)] if $\rft(v)$ lies on the right arm of $T$, then $\theta_v(T)$ is obtained by cutting $v$ off and inserting it into the right arm of $T$ so that the parent node of $\rft(v)$ is $v$;
\item[(b)] if $\rft(v)$ lies on the right chain $C_w$ of $T$ for some $w$,
then $\rft(v)$ is the leading element of $C_w$ and $\theta_v(T)$ is obtained by inserting $v$ into $C_w$ such that the parent node of $\rft(v)$ is $v$;
\item[(c)] if $\rft(v)$ does not exist,  then $\theta_v(T)$ is obtained by cutting  $v$ and attaching it as right child of the last node of the right arm of $T$.
\end{itemize} 
For any $321$-avoiding tree $T$, we can define 
$$\Theta(T)=\theta_{v_m}\circ\theta_{v_{m-1}}\circ\cdots\circ\theta_{v_1}(T),$$
assuming $\RL(T)=\{v_1<v_2<\cdots<v_m\}$. 
See Fig.~\ref{fig action} for an example of $\Theta(T)$.

For a $321$-avoiding tree $T\in \LB_n$ and a node  $v\in\RA(T)$, let 
$$
\mathsf{g}_v(T):=|\{b>v: b\notin T_r(v)\}|.
$$
For example,  if $T$ is the tree in Fig.~\ref{fig:321-tree}, then $\mathsf{g}_1(T)=|\{2,4,5\}|=3$ and $\mathsf{g}_3(T)=|\{4,5,7,8,9\}|=5$. Indeed, if $\pi=\lambda^{-1}(T)$ and $v=\pi_j\in\Rlmi(\pi)$, then $\mathsf{g}_v(T)=\lw_j(\pi)$. Introduce the parameter of $T$ as 
$$
\mathsf{g}(T):= \max_{v\in{\RA(T)}}\{\mathsf{g}_v(T)\}.
$$
Thus, 
\begin{equation}\label{g:mlw}
\mathsf{g}(T)=\mlw(\lambda^{-1}(T)).
\end{equation}
The following observation is an immediate consequence of the construction of $\Theta(T)$.

\begin{figure}
\centering
\begin{tikzpicture}[scale=0.17]
\draw[-](0,20) to (30,2);
\draw[-](0,20) to (-10,15);
\draw[-](5,17) to (1,14);
\draw[-](20,8) to (15,5);
\draw[-](25,5) to (20,2);
\draw[-](23,-2) to (20,2);
\draw[-](-10,15) to (-4,5);
\draw[-](1,14) to (7,4);

\node at  (0,20){$\bullet$};\node at  (1,21){$1$};
\node at  (5,17){$\bullet$};\node at  (6,18){$3$};
\node at  (10,14){$\bullet$};\node at  (11,15){$6$};
\node at  (15,11){$\bullet$};\node at  (16,12){$10$};
\node at  (20,8){$\bullet$};\node at  (21,9){$11$};
\node at  (25,5){$\bullet$};\node at  (26,6){$13$};
\node at  (30,2){$\bullet$};\node at  (31,3){$15$};
\node at  (-10,15){$\bullet$};\node at  (-11,16){$2$};
\node at  (-7,10){$\bullet$};\node at  (-8,9){$4$};
\node at  (-4,5){\red{$\bullet$}};\node at  (-5,4){$5$};
\node at  (1,14){$\bullet$};\node at  (0,15){$7$};
\node at  (4,9){$\bullet$};\node at  (3,8){$8$};
\node at  (7,4){\red{$\bullet$}};\node at  (6,3){$9$};
\node at  (15,5){\red{$\bullet$}};\node at  (14,4){$12$};
\node at  (23,-2){\red{$\bullet$}};\node at  (24,-3){$16$};
\node at  (20,2){$\bullet$};\node at  (19,1){$14$};

\draw[->,dashed,blue,thick] (-3.5,5.5) to (0.5,13.5);
\draw[->,dashed,blue,thick] (7.5,4.5) to (12,12);
\draw[->,dashed,blue,thick] (15.5,5) to (21.5,6);
\draw[->,dashed,blue,thick] (23.5,-1.5) to (31,0);

\node at  (10,-4){$\downarrow$};\node at  (11.5,-4){$\Theta$};
\draw[-](-3,-8) to (33,-26);
\draw[-](-3,-8) to (-11,-12);
\draw[-](-8,-15) to (-11,-12);
\draw[-](1,-10) to (-2,-13);
\draw[-](4,-19) to (-2,-13);
\draw[-](25,-22) to (22,-24);

\node at  (-3,-8){$\bullet$};\node at  (-2,-7){$1$};
\node at  (1,-10){$\bullet$};\node at  (2,-9){$3$};
\node at  (5,-12){$\bullet$};\node at  (6,-11){$6$};
\node at  (9,-14){\red{$\bullet$}};\node at  (10,-13){$9$};
\node at  (13,-16){$\bullet$};\node at  (14,-15){$10$};
\node at  (17,-18){$\bullet$};\node at  (18,-17){$11$};
\node at  (21,-20){\red{$\bullet$}};\node at  (22,-19){$12$};
\node at  (25,-22){$\bullet$};\node at  (26,-21){$13$};
\node at  (29,-24){$\bullet$};\node at  (30,-23){$15$};
\node at  (33,-26){\red{$\bullet$}};\node at  (34,-25){$16$};
\node at  (22,-24){$\bullet$};\node at  (21,-25.5){$14$};
\node at  (-11,-12){$\bullet$};\node at  (-12,-13){$2$};
\node at  (-8,-15){$\bullet$};\node at  (-9,-16){$4$};
\node at  (-2,-13){\red{$\bullet$}};\node at  (-3,-14){$5$};
\node at  (1,-16){$\bullet$};\node at  (0,-17){$7$};
\node at  (4,-19){$\bullet$};\node at  (3,-20){$8$};
\end{tikzpicture}
\caption{An example of $\Theta(T)$, where $T\in\LB_{16}$.\label{fig action}}
\end{figure}
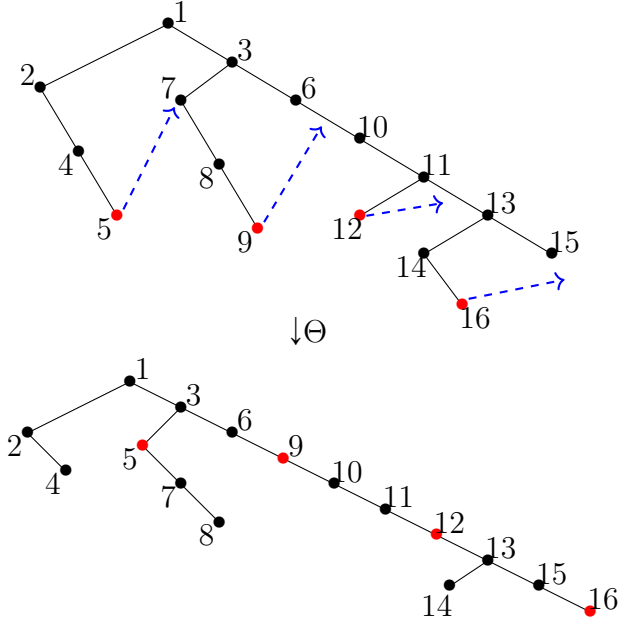

\begin{observation}\label{prop:cha}
For any $T\in\LB_n$ with $\mathsf{g}(T)>0$, we have  $\Theta(T)\in \LB_n$  and 
 $$
 \mathsf{g}(\Theta(T))=\mathsf{g}(T)-1.
 $$
\end{observation}
\begin{proof}
 The difference between $\mathsf{g}_v(T)$ and $\mathsf{g}_v(\Theta(T))$ is caused by the removal of the largest node in $\{b>v: b\notin T_r(v)\}$.
\end{proof}

\begin{lemma}\label{lem sort}
For any $\pi\in \S_n(321)$, we have $\s(\pi)=\lambda^{-1}\circ\Theta\circ\lambda(\pi)$.
\end{lemma}

\begin{proof}
According to Proposition~\ref{def:dyn}, we can use the dynamic description  of the stack-sorting operator  $\s$ to show that $\s(\pi)=\lambda^{-1}\circ\Theta\circ\lambda(\pi)$.
By Observation~\ref{lem:bijpro}, if  $\DES(\pi)=\{b_1<b_2<\cdots<b_k\}$, then $\RL(T)=\{b_1<b_2<\cdots<b_k\}$ with $T=\lambda(\pi)$. By Proposition~\ref{def:dyn}, we have 
$$
\s(\pi)=\s_{b_k}\circ\s_{b_{k-1}}\circ\cdots\circ\s_{b_1}(\pi).
$$
On the other hand, 
$$\Theta(T)=\theta_{b_k}\circ\theta_{b_{k-1}}\circ\cdots\circ\theta_{b_1}(T).$$
Since $\rfp(b_i)=\rft(b_i)$ (if it exists) for each $i$, it is straightforward to verify that $\s_{b_1}(\pi)=\lambda^{-1}(\theta_{b_1}(T))$, $\s_{b_2}\circ\s_{b_1}(\pi)=\lambda^{-1}(\theta_{b_2}\circ\theta_{b_1}(T))$, and then so on, and finally  $\s_{b_k}\circ\s_{b_{k-1}}\circ\cdots\circ\s_{b_1}(\pi)=\lambda^{-1}(\theta_{b_k}\circ\theta_{b_{k-1}}\circ\cdots\circ\theta_{b_1}(T))$, as desired. 
\end{proof}

We can now prove Lemma~\ref{lem:321}. 

\begin{proof}[{\bf Proof of Lemma~\ref{lem:321}}]
This follows directly from  relationship~\eqref{g:mlw}, Observation~\ref{prop:cha} and Lemma~\ref{lem sort}. 
\end{proof}

\section{Further applications}
\label{sec:4}
As applications of Lemmas~\ref{lem:213} and~\ref{lem:321}, we can characterize  $213$- or $321$-avoiding $t$-stack-sortable permutations using the pattern $23\cdots(t+2)1$. 

\begin{proposition}\label{prop lrrp}
      For any positive integer $t$ and permutation $\pi\in\S_n(213)$, the following are equivalent:
   \begin{enumerate}
       \item $\pi$ avoids the pattern $23\cdots(t+2)1$;
       \item $\lrrp(\lambda(\pi))\leq t$.
   \end{enumerate}
   Consequently, 
   \begin{equation}\label{pat:t213}
   \S_n^t(213)=\S_n(231,23\cdots(t+2)1). 
  \end{equation}
\end{proposition}
 
\begin{proof}
    Let $\pi\in \S_n(213)$ with $\lrrp(\lambda(\pi))> t$. Assume that  $\lambda(\pi)$ is equipped with natural labelings (i.e., nodes labeled according to right pre-order).  There exists a  restricted  path $P$ in $\lambda(\pi)$ containing exactly  $t$ right edges. Suppose that $u_1<u_2<\cdots< u_t$ are the parent nodes of these $t$ right edges  and $w$ is the maximal node in $P$.
    Since $P$ is node-disjoint from the right arm, there exists a node $v$ on the right arm such that $P$ lies entirely in the left subtree of $v$.
The subsequence $u_1,\ldots, u_t, w, v$ form a pattern $23\cdots(t+2)1$ in $\pi$. 

    Conversely, suppose that $\pi \in \S_n(213)$ contains a pattern $23\cdots(t+2)1$. Let $\pi_{i_1}\pi_{i_2}\cdots\pi_{i_{t+2}}$ be a subsequence forming this pattern. 
  For each $1 \leq j \leq t$, since $\pi$ avoids $213$, all entries between $\pi_{i_j}$ and $\pi_{i_{j+1}}$ in $\pi$ must be greater than $\pi_{i_j}$. 
    By the definition of $\lambda$, $\pi_{i_{j+1}}$ must appear in the right subtree of $\pi_{i_j}$ in $\lambda(\pi)$ and the path from $\pi_{i_j}$ to $\pi_{i_{j+1}}$ contains at least one right edge. Therefore, the path from $\pi_{i_1}$ to $\pi_{i_{t+1}}$ contains at least $t$ right edges. 
    Moreover, since $\pi_{i_{t+2}}$ appears to the right of $\pi_{i_j}$ in $\pi$ and $\pi_{i_{t+2}} < \pi_{i_j}$, $\pi_{i_j}$ does not lie on the right arm of the tree. Thus, this path is node-disjoint from the right arm, proving that $\lrrp(\lambda(\pi)) > t$. This proves the first statement. 
    
    The second statement then follows by combining the first statement with Lemma~\ref{lem:213}. 
\end{proof}

\begin{proposition}\label{lem pattern}
For any positive integer $t$ and permutation $\pi \in \S_n(321)$, the following  are equivalent:
\begin{enumerate}
\item $\pi$ avoids the pattern $23\cdots(t+2)1$;
\item $\mlw(\pi)\leq t$.
\end{enumerate}
Consequently, 
   \begin{equation}\label{pat:t321}
   \S_n^t(321)=\S_n(321,23\cdots(t+2)1). 
  \end{equation}
\end{proposition}

\begin{proof}
If $\mlw(\pi)>t$, then there exists $\pi_j\in\Rlmi(\pi)$ such that $\lw_j(\pi)=j-\pi_j>t$. Thus, there are at least $t+1$ letters to the left of $\pi_j$ and  greater than $\pi_j$. Since $\pi$ is $321$-avoiding, these $t+1$ letters are increasing from left to right, which together with $\pi_j$ form a $23\cdots(t+2)1$ pattern.  

Conversely, if $\pi$ contains the pattern $23\cdots(t+2)1$ with $\pi_j$ playing the role of $1$, then $\pi_j$ must be a left-to-right minimum  and $j-\pi_j\geq t+1$. This implies that $\mlw(\pi)>t$, which completes the proof of the first statement. 

The second statement follows by combining the  first statement with Lemma~\ref{lem:321}. 
\end{proof}

At the end of~\cite{Zhang},  Zhang and Kitaev posed another enumerative conjecture as follows. 

\begin{conjecture}[Zhang and Kitaev~\cite{Zhang}]\label{main conj2}
    For any $t\geq1$ and $p\in\{213,321\}$,  $t$-stack-sortable $p$-avoiding  permutations are in one-to-one correspondence with  $(132,12\cdots (t+2))$-avoiding permutations.
\end{conjecture}

For any permutation $\pi=\pi_1\cdots\pi_n$, define two fundamental symmetry operations:
\begin{itemize}
    \item  $\r(\pi)=\pi_n\cdots\pi_1$, the reversal of $\pi$;
    \item   $\c(\pi)=(n+1-\pi_1)\cdots(n+1-\pi_n)$, the complement of $\pi$.
\end{itemize}
The composition $\r\circ\c$ sets up a one-to-one correspondence   between $\S_n(132,12\cdots (t+2))$ and $\S_n(213,12\cdots (t+2))$. Thus, in order to prove Conjecture~\ref{main conj2}, it suffices to construct a bijection between $\S_n(213,12\cdots (t+2))$ and $\S_n(213,23\cdots (t+2)1)$. Such a simple bijection will be constructed using binary trees. 

For a binary tree $T$, a {\em longest right path} of $T$ is a path containing maximum number of right edges. Denote by $\lrp(T)$ the number of right edges in any longest right path of $T$. We have the following characterization of $(213,12\cdots (t+2))$-avoiding permutations.

\begin{proposition}\label{prop lrp}
      For any $t\geq1$ and $\pi\in\S_n(213)$, the following are equivalent:
   \begin{enumerate}
       \item $\pi$ avoids the pattern $12\cdots (t+2)$;
       \item $\lrp(\lambda(\pi))\leq t$.
   \end{enumerate}
\end{proposition}
\begin{proof}
   The proof follows by similar arguments to those in Proposition~\ref{prop lrrp}.
\end{proof}


 We can now provide a bijective proof of Conjecture~\ref{main conj2}.
\begin{proof}[{\bf Bijective proof of Conjecture~\ref{main conj2}}]
In view of Propositions~\ref{prop lrrp} and~\ref{prop lrp}, the composition 
$\lambda^{-1}\circ\beta\circ\gamma\circ\lambda$
sets up a one-to-one correspondence between  $\S_n(321,23\cdots(t+2)1)$ and 
$\S_n(321,123\cdots(t+2))$.
   \end{proof}

\begin{remark}
Using the transfer-matrix method, Chow and West~\cite{CW} (see also~\cite{Man}) proved that 
$$
|\S_n(321,23\cdots(t+2)1)|=|\S_n(231,23\cdots(t+2)1)|=|\S_n(231,12\cdots(t+2))|.
$$
Our approach using binary trees provides bijective proofs of the above equinumerosities with refinements. 
\end{remark}

\section*{Acknowledgement} 
This work was
initiated while the second author was visiting the third author at Research Center for Mathematics and Interdisciplinary Sciences of Shandong University (Qingdao) in June, 
2025. This work was supported by the National  Science Foundation of China (grants 12322115 \& 12271301) 
 and the Fundamental Research Funds for the Central Universities.

\end{document}